\documentclass[12pt]{article}
\usepackage{color,latexsym,amsmath,amssymb,amsthm,path,url,enumerate,hyperref, graphicx}
\usepackage{hyperref}
\usepackage{caption}
\usepackage{subcaption}

\newtheorem{theorem}{Theorem}[section]

\newcommand{\R}{\mathbb{R}}

\newcommand{\F}{\mathcal{F}}

\renewcommand{\leq}{\leqslant}
\renewcommand{\ge}{\geqslant}
\renewcommand{\le}{\leqslant}

\def\cref#1{Corollary~$\ref{#1}$}

\def\b1{\bar{1}}
\def\cb1{\cdot \bar{1}}

\usepackage{array,color,colortbl}

\newcommand{\T}{\mathcal{T}}

\title{Bounds on  piercing and line-piercing numbers in families of convex sets in the plane}

\author{Shira Zerbib \thanks{Department of Mathematics, Iowa State University,  USA. Email: zerbib@iastate.edu. The author was supported by NSF grant DMS-1953929.}}
\date{}
\begin{document}
\maketitle 
%%%%%%%%%%%%%%%%%%%%%%%%%%%%%%%%%%%%%%%%%%%%%%%%%%%%%%%%%%%%%%%%%%%%%%%%%%%%%%%%%

\begin{abstract}
A family of sets has the $(p, q)$ property if
among any $p$ members of it some $q$ intersect.
It is shown that  if a finite family of compact convex sets in $\R^2$ has the $(p+1,2)$ property then it is pierced by $\lfloor \frac{p}{2} \rfloor +1$ lines.  A colorful version of this result is proved as well. As a corollary, the following is proved: Let 
$\F$ be a finite family of compact  convex sets in the plane with no isolated sets, and let $\F'$ be the family of its pairwise intersections. If $\F$ has the $(p+1,2)$ property and $\F'$ has   the $(r+1,2)$ property, then $\F$ is pierced by $(\lfloor \frac{r}{2} \rfloor ^2 +\lfloor\frac{r}{2} \rfloor)p$ points when $r\ge 2$, and by $p$ points otherwise. The proofs use the topological KKM theorem.
\end{abstract}

\section{Introduction}

For  integers $p \ge q>0$, a family of sets $\F$ has the {\em $(p, q)$ property} if
among any $p$ members of $\F$ some $q$ have a nonempty intersection.
The {\em matching number} of $\F$ is the largest size a subfamily of pairwise disjoint sets in $\F$. Note that $\F$ has the $(p+1,2)$ property if and only if its matching number is at most $p$.
We say that  $\F$ is {\em pierced  by $k$ points} if there exists a set $C\subseteq \bigcup \F$, such that $|C|=k$ and  every $F\in \F $ contains a point from $C$.  
The {\em piercing number} of $\F$, denoted by $\tau(\F)$, is the smallest $k$ such that $\F$ is pierced by $k$ points.

The classical theorem of Helly \cite{helly} states that any family of compact
convex sets in $\R^d$ with the $(d+ 1, d+ 1)$-property is pierced by one point.
Hadwiger and Debrunner \cite{HD} considered the more general problem of studying
the piercing numbers of families  of compact, convex sets in $\R^d$ that
satisfy the $(p, q)$-property. They proved   that whenever $p\ge q\ge d+1$ and  $p(d-1)< d(q-1)$,  any family of compact, convex sets in $\R^d$ is pierced by $p-q+1$ points. They also conjectured  that for every  $p\ge q\ge d+1$ there exists a constant $c=c(p,q;d)$ such that any family of compact, convex sets in $\R^d$ with the $(p,q)$ property is pierced by $c$ points. This conjecture  was proved by Alon and Kleitman in 1992 \cite{AK}.

It is known that such a constant $c$ does not exist when $q \le d$. For example, when $d=q=2$, a family of pairwise intersecting line segments in general position 
in the plane has piercing number at least $|\F|/2$, which is not bounded by a constant. However, such a constant may exist if the family is further restricted. For example, Danzer \cite{Danzer} showed that a family of disks in the plane with the $(2,2)$ property has piercing number at most 4;
another example is 
a well-known conjecture of Wegner from 1965 \cite{wegner}, asserting that a family of axis-parallel rectangles in the plane with the $(p+1,2)$ property is pierced by $2p$ points (where the best current known upper bound, proven by Correa, Feuilloley,  P\'erez-Lantero, and Soto \cite{soto} is $\tau\le O(p)(\log\log p)^2$ ). 

The question of finding or improving bounds on the piercing numbers in general families of convex sets in $\R^d$ with the $(p,q)$ property for $q\ge d+1$, or in restricted families when $q$ may be smaller than $d+1$, has received great attention over the years, see e.g., \cite{Eckhoff, KST,  mcginnis} for general families, and \cite{CSZ, GZ, karasev, KT} for restricted families.

In this  note we give a new bound for the piercing numbers in general families $\F$ of compact convex sets in $\R^2$ with the $(p+1,2)$ property (namely, when $q=d=2$), where the bound is given in terms of the matching number of the family $\F'=\{S\cap S' \mid S,S'\in \F\}$ of pairwise intersections in $\F$. 

An {\em isolated set} is a set in $\F$ that does not intersect any other set in $\F$. 

\begin{theorem}\label{main}
 Let $p,r$ be positive integers. Let $\F$ be a finite family of compact convex sets in $\R^2$ with no isolated sets, and let $\F'$ be the family of its pairwise intersections. If $\F$ has the $(p+1,2)$ property and $\F'$ has the $(r+1,2)$ property then $\tau(\F)\le (\lfloor \frac{r}{2} \rfloor ^2 +\lfloor\frac{r}{2} \rfloor)p$ when $r\ge 2$, and $\tau(\F)\le p$ otherwise.
\end{theorem}

Note that the condition that $\F$ has no isolated sets does not pose a significant restriction on the families $\F$ we may consider in the theorem, as any isolated set would add exactly 1 to both the matching number and  the piercing number of $\F$.

When $p=r=1$, Theorem \ref{main} implies Helly's theorem in the plane, as the condition that both $\F$ and $\F'$ have the $(2,2)$ property is equivalent to the condition that  $\F$ has the $(3,3)$ property, by Radon's theorem \cite{radon}. 

Moreover, if $p=1$ and $r\ge 2$, then the conditions $\F$ has the $(p+1,2)$ property and $\F'$ has the $(r+1,2)$ property imply that $\F$ has the $(t,3)$ property for some $t$ satisfying $r+1 \le {t \choose 2}$. Thus, by the Alon-Kleitman theorem, there exists some constant $c=c(t,3;2)$ such that $\tau(\F)\le c$. The current best known bound on this constant,  proved in \cite{KST},  is $c(t,3;2) \le O(t^4)$.
Theorem \ref{main} gives a specific bound rather than an order estimation,  under the stronger condition that $\F'$ has the $(r+1,2)$ property.

Similarly, for any $p,r\ge 2$, the conditions $\F$ has the $(p+1,2)$ property and $\F'$ has the $(r+1,2)$ property  imply that $\F$ has the $(t,4)$ property, for some $t\le 2r+p+1$. Indeed, let $\T$ be a subfamily of $\F$ of size $2r+p+1$. At each step pick two sets that intersect and remove them, until less than $p+1$ sets remain. We get $r+1$ intersecting pairs, and by the $(r+1,2)$ property of $\F'$ some two of them intersect. 
The bound proven in \cite{KST} then implies $\tau(\F)\le O((2r+p)^3)$. Our bound, under the stronger conditions of Theorem \ref{main}, provides a  specific bound rather than an order estimation, and an improvement in order when $p>>r$ or $r>>p$.

\medskip

The proof of  Theorem \ref{main} follows from a bound on the line-piercing numbers in families of convex sets in the plane. We say that a family $\F$ of sets in $\R^d$ is {\em pierced by $k$ lines}, if there exist $k$ lines in $\R^d$ whose union intersects every set in $\F$. The {\em line-piercing number} of $\F$ is the smallest $k$ such that $\F$ is pierced by $k$ lines. We say that $\F$ has {\em property $T(r)$} if every $r$ or fewer sets in $\F$ are pierced by a line.

The problem of bounding the line-piercing numbers of families of convex sets in the plane with the $T(r)$ property has been studied since the 1940's. In 1969 Eckhoff \cite{Eckhoff1} proved that if a family of compact convex sets  has the $T(k)$ property, for $k\ge 4$, then it is pierced by two lines. 
In 1993 
Eckhoff \cite{Eckhoff3}  proved that a family of compact convex sets  with the $T(3)$ property is pierced by 4 lines, and conjectured that  this bound can be improved to 3. This conjecture was proved in \cite{MZ}. 

If $\F$ is a finite family of compact convex sets in $\R^2$ with the $(p+1,2)$ property, then $\F$ is pierced by $p$ lines; this follows from projecting $\F$ onto a line and using a theorem of Gallai (Theorem \ref{gallai} below), asserting that the matching and piercing numbers in families of intervals in $\R$ are equal.  
Here we obtain a tight bound on the line-piercing numbers of such families: 

\begin{theorem}\label{main2}
If $\F$ is a finite family of compact convex sets in $\R^2$ with the $(p+1,2)$ property then it is pierced by $\lfloor \frac{p}{2} \rfloor+1$ lines. 
\end{theorem}

The bound in Theorem \ref{main2} is tight by the following example. Let $\F$ be the family of edges of a regular $(2p+1)$-gon. Then $\F$ has the $(p+1)$-property, and since every line in $\R^2$ intersects at most 4 elements of $\F$, at least $\lceil\frac{2p+1}{4}\rceil = \lfloor \frac{p}{2} \rfloor+1$ lines are needed to pierce $\F$.   
\medskip

We shall also prove the following colorful version of Theorem \ref{main2}: 
\begin{theorem}\label{main3}
Suppose $\F_1,\dots,\F_{2k}$ are finite families of compact convex sets in $\R^2$, such that for any choice of sets $F_i \in \F_i, ~i\in [2k]$, some two sets among $F_1,\dots, F_{2k}$ intersect.   Then there exists $i\in [2k]$ such that $\F_i$ is pierced by $k$ lines. 
\end{theorem}

The proof of Theorem \ref{main2} uses the topological KKM theorem, due to Knaster, Kuratowski and Mazurkiewicz \cite{kkm}. Let $$\Delta_{k} =\Big\{ (x_1,\dots,x_{k+1}) \mid \sum_{i=1}^{k+1} x_i=0, x_i\ge 0 ~\text{ for all 
 } i\in [k+1]\Big\}$$  be the standard $k$-dimensional simplex in $\R^{k+1}$. For a face $\sigma$ of $\Delta_{k}$, we write $i\in \sigma$ if $\sigma$ contains the $i$-th vertex of $\Delta_k$, namely the point $(x_1,\dots,x_{k+1})$ with $x_i=1$ and $x_j=0 $ for $i\neq j$. Note that if $x_i>0$ for some point $(x_1,\dots,x_{k+1}) \in \sigma$ then $i\in \sigma$.

\begin{theorem}[The KKM theorem]\label{kkm}
If  open subsets $A_1,\dots,A_{k+1}$ of the $k$-dimensional simplex $\Delta_{k}$ 
satisfy $\sigma \subseteq \bigcup_{i \in \sigma} A_i$ for 
every face $\sigma$ of $\Delta_{k}$, then $\bigcap_{i=1}^{k+1} A_i \neq \emptyset$. 
\end{theorem}

A family  $(A_1,\dots, A_{k+1})$ of   subsets of $\Delta_k$ satisfying the conditions of Theorem \ref{kkm} is called a {\em KKM cover} of $\Delta_k$. Let $S_n$ be the group of permutations of the elements in $[n]=\{1,2,\dots,n\}$. 
The proof of Theorem \ref{main3}  will use a colorful generalization of the KKM theorem, due to Gale \cite{gale}. 

\begin{theorem}[The colorful KKM theorem]\label{colorfulkkm}
If  $(A^i_1,\dots,A^i_{k+1}), i\in [k+1]$, are $k+1$ KKM covers of $\Delta_{k}$,  
 then there exists a permutation $\pi \in S_{k+1}$ such that $\bigcap_{i=1}^{k+1} A^{\pi(i)}_i \neq \emptyset$. 
\end{theorem}

The application of Theorems \ref{kkm} and \ref{colorfulkkm} in the proof of Theorems \ref{main2} and \ref{main3} is similar to their application in \cite{mcginnis, MZ}: in all cases one  uses a simplex of appropriate dimension to model the configuration space of some special sets of lines in the plane, and then defines a KKM cover corresponding to certain regions, obtained from the partition of the plane by the lines.  Those ideas  appeared first in \cite{mcginnis, MZ}.  The new ingredient in the proof of Theorems \ref{main2}  and \ref{main3} is the new way the regions are defined: here we define the regions inductively, and they are pairwise disjoint, which allows us to consider any number of distributed lines, and therefore to prove the theorem for all $p$. This shows the robustness of this KKM-based method, which was first introduced in \cite{mcginnis, MZ}: changing the way the regions are defined allows for more applications of the method. 

For the proof of  Theorem \ref{main} we will also use the following theorem of Gallai \cite{gallai}, which is a special case  of the theorem of Hadwiger and Debrunner mentioned above.

    \begin{theorem}\label{gallai} 
      If $\F$ a finite family of compact intervals in $\R$ with the \linebreak$(p+1,2)$ property, then $\tau(\F)\le p$.
        \end{theorem}

Finally, we will use a theorem bounding the piercing numbers in families of $d$-intervals. 
A {\em (separated) $d$-interval} is a union of at most $d$ compact (possibly empty) intervals, one  on each of $d$ copies of $\R$. The following was proved by Tardos \cite{tardos} for $d=2$ and by Kaiser \cite{kaiser} for all $d\ge 2$.
        
        \begin{theorem}\label{dintervals} 
        Let $d\ge 2$. If $\F$ a finite family of compact $d$-intervals with the $(p+1,2)$ property, then $\tau(\F) \le (d^2-d)p$.
        \end{theorem}

\section{Proofs of Theorems \ref{main2} and \ref{main3}}

\begin{proof}[Proof of Theorem \ref{main2}]
Since $\F$ is finite and the sets in it are compact, we may assume that all the sets in $\F$ are    contained in the open unit disk $D$.  Let $f(t)$ be a parameterization of $U=\partial D$ defined by $f(t)=(\textrm{cos}(2\pi t), \textrm{sin}(2\pi t))$. 

Let $k=\lfloor \frac{p}{2}\rfloor$. 
A point $x=(x_1,\dots,x_{2k+2})\in \Delta_{2k+1}$ corresponds to $2k+2$ points on $U$, given by $f_i(x)=f(\sum_{j=0}^i x_{j})$ for $0\leq i\leq 2k+1$, where $x_0=0$.  
For $i=0,\dots, k$, let $\ell_i(x)=\ell_{i+k+1}(x)$ be the line segment connecting $f_i(x)$ and $f_{i+k+1}(x)$ (it may happen that $\ell_i(x)$ is a single point).
Let $L(x)=\bigcup_{i=0}^k \ell_i(x)$.
For $i\in [2k+2]$, let $P_i(x)$ be  the open region  bounded by the  $\ell_i(x)$,  $\ell_{i-1}(x)$  and the arc between $f_{i-1}(x)$ and $f_i(x)$. Further, for $i\in [2k+2]$, define recursively the regions $T_i(x)$ as follows: $T_1(x)=P_1(x)$, and $T_i(x) = P_i(x) \setminus \bigcup_{j=1}^{i-1} T_j(x)$ for all $2\le i \le 2k+2$. Finally, let $R_i(x)=T_i(x) \setminus L(x)$. Note that since $L(x)\cup (\bigcup_{i=1}^{2k+2} P_i(x)) = D-U$, we  have  $L(x)\cup (\bigcup_{i=1}^{2k+2} R_i(x)) = D-U$. Moreover, the regions $R_1(x),\dots, R_{2k+2}(x)$ are pairwise disjoint and open 
 (see Figure \ref{figure1}).  Notice that $R_i(x)=\emptyset$ when $x_{i}=0$, because in this case the arc between $f_{i-1}(x)$ and $f_i(x)$ is of length 0, and thus $P_i(x)=\emptyset$. 

\begin{figure}
     \centering
     \begin{subfigure}[b]{0.465\textwidth}
         \centering
         \includegraphics[width=\textwidth]{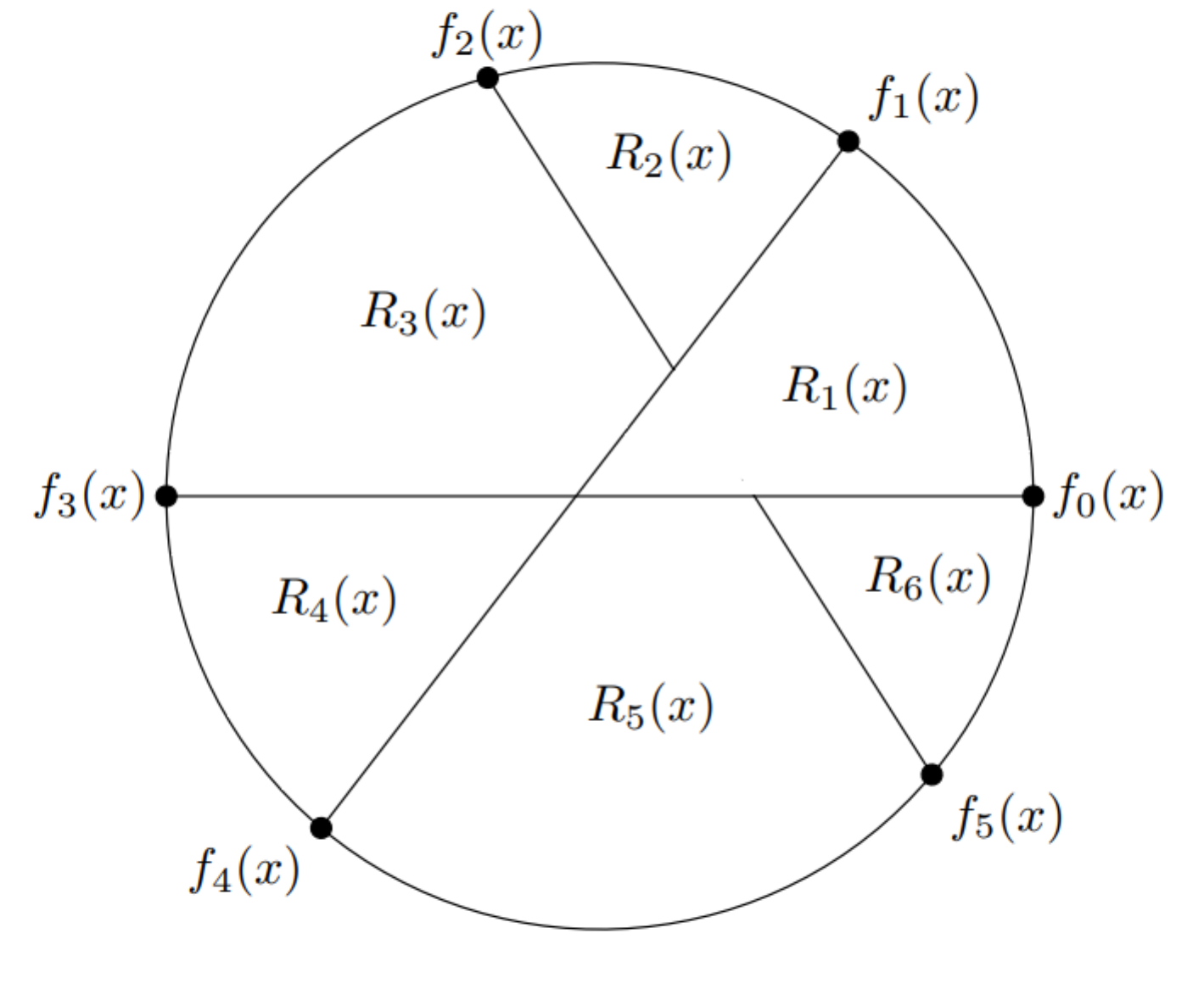}
         \caption{$p=4$ or $p=5$}
     \end{subfigure}
     \hfill
     \begin{subfigure}[b]{0.45\textwidth}
         \centering
         \includegraphics[width=\textwidth]{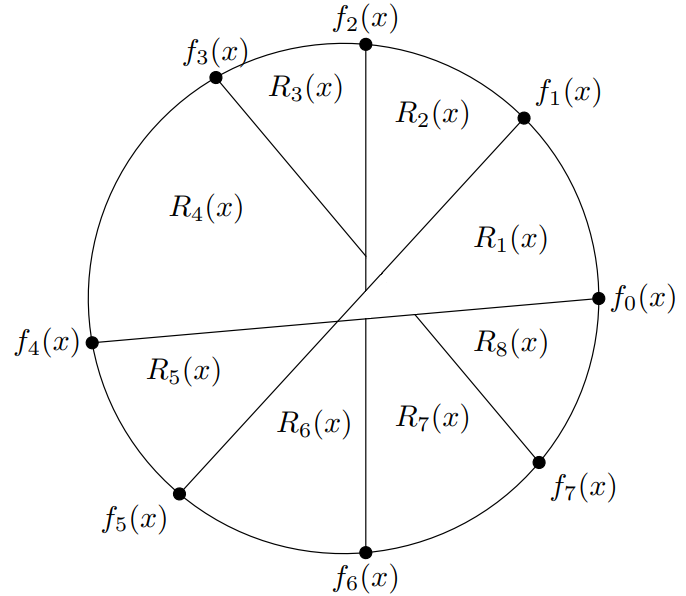}
         \caption{$p=6$ or $p=7$}
     \end{subfigure}
        \caption{The regions the region $R_i(x)$ (note that $R_i(x)$ does not contain points from the lines $\ell_j(x)$ for any $j$, and thus it is not necessarily connected. However, to make the picture clear, we depicted the regions $R_i(x)$ as connected regions).   }
        \label{figure1}
\end{figure}

For every $i\in [2k+2]$, define  sets $A_i \subseteq \Delta_{2k+1}$ as follows:  $x \in A_i$ if and only if there exists a set $S \in \F$ so that $S$ is contained in $R_i(x)$. 

Assume for contradiction that $\Delta_{2k+1} = \bigcup_{i=1}^{2k+2} A_i$. We claim that the family of sets $(A_1,\dots, A_{2k+2})$ forms a KKM cover of $\Delta_{2k+1}$. First, the sets $A_i$ are open because the sets in $\F$ are closed. Moreover, if $x \in \sigma $ for some face $\sigma$ of  $\Delta_{2k+1}$, then for any $i\notin \sigma$ we have $x_i=0$ and thus $R_i(x)=\emptyset$, so $R_i(x)$ cannot contain a set of $\F$. Therefore $x\notin A_i$ when $i\notin \sigma$. Since $\Delta_{2k+1} = \bigcup_{i=1}^{2k+2} A_i$ by our  assumption, it follows that $x\in A_i$ for some $i\in \sigma$, showing $\sigma \subseteq \bigcup_{i\in \sigma} A_i$.
Hence, by the KKM theorem, there exists $x \in \bigcap_{i=1}^{2k+2} A_i$, and thus for every $i\in [2k+2]$ there exists a set $F_i \in \F$ with $F_i \subset R_i(x)$. But then the sets $F_1,\dots,F_{2k+2}$ are pairwise disjoint, and since $2k+2 \ge p+1$, this is a  contradiction to the $(p+1,2)$ property of $\F$.
 
We conclude that there exists $x=(x_1,\dots,x_{2k+2}) \in \Delta_{2k+1} \setminus (\bigcup_{i=1}^{2k+2} A_i)$. Since $x\notin A_i$, there is no set in $\F$ lying in $R_i(x)$, for all $i\in [2k+2]$. But then  $L(x)\cup (\bigcup_{i=1}^{2k+2} R_i(x)) = D-U$ and the fact that all the sets in $\F$ lie in $D-U$  entail that every 
set in $\F$ intersects one of the $k+1$ lines $\ell_i(x)$, $0\le i \le k$. 
This completes the proof of the theorem.
\end{proof}

\begin{proof}[Proof of Theorem \ref{main3}]
The proof follows the lines of the previous proof, with some modifications. 
Again,  we may assume that all the sets in $\bigcup_{j=1}^{2k}\F_j$ are  contained in the open unit disk $D$.  Let $f(t)=(\textrm{cos}(2\pi t), \textrm{sin}(2\pi t))$.

A point $x=(x_1,\dots,x_{2k})\in \Delta_{2k-1}$ corresponds to $2k$ points on $U$, given by $f_i(x)=f(\sum_{j=0}^i x_{j})$ for $0\leq i\leq 2k-1$, where $x_0=0$.  
For $i=0,\dots, k-1$, let $\ell_i(x)=\ell_{i+k}(x)$ be the line segment connecting $f_i(x)$ and $f_{i+k}(x)$.
Let $L(x)=\bigcup_{i=0}^{k-1} \ell_i(x)$.
As before, for $i\in [2k]$, let $P_i(x)$ be  the open region  bounded by the lines $\ell_i(x)$ and  $\ell_{i-1}(x)$  and the arc between $f_{i-1}(x)$ and $f_i(x)$. Further, for $i\in [2k]$, define recursively the regions $T_i(x)$ as before: $T_1(x)=P_1(x)$, and $T_i(x) = P_i(x) \setminus \bigcup_{j=1}^{i-1} T_j(x)$ for all $2\le i \le 2k$. Finally, let $R_i(x)=T_i(x) \setminus L(x)$. Note that   $L(x)\cup (\bigcup_{i=1}^{2k} R_i(x)) = D-U$, the regions $R_1(x),\dots, R_{2k}(x)$ are pairwise disjoint and open, and  $R_i(x)=\emptyset$ when $x_{i}=0$.

For every $i,j\in [2k]$, define  sets $A^j_i \subseteq \Delta_{2k-1}$ as follows:  $x \in A^j_i$ if and only if there exists a set $S \in \F_j$ so that $S$ is contained in $R_i(x)$. 

Assume for contradiction that for every $j\in [2k]$ we have $\Delta_{2k-1} = \bigcup_{i=1}^{2k} A^j_i$. Then by the same argument as in the previous proof,  for every $j\in [2k]$, the family $(A_1^j,\dots,A^j_{2k})$ is a KKM cover of $\Delta_{2k-1}$.

Hence, by the colorful KKM theorem, there exists $\pi\in S_{2k}$ such that $\bigcap_{i=1}^{2k} A^{\pi(i)}_i \neq \emptyset$. Let $x\in \bigcap_{i=1}^{2k} A^{\pi(i)}_i$. By the definition of the sets $A^j_i$,  for every $i\in [2k]$ there exists a set $F_i \in \F_{\pi(i)}$ satisfying $F_i \subset R_i(x)$. But then the sets $F_1,\dots,F_{2k}$ are pairwise disjoint,  violating  the condition of the theorem.
 
We conclude that there exists $j\in [2k]$ and $x\in \Delta_{2k-1}$ such that 
$x \in \Delta_{2k-1} \setminus (\bigcup_{i=1}^{2k} A^j_i)$. By the same argument as before, this means that $\F_j$ is pierced by the $k$ 
lines $\ell_i(x)$, $0\le i \le k-1$. 
This completes the proof of the theorem.
\end{proof}

\section{Proof of Theorem \ref{main}}
We apply Theorem \ref{main2} with the family  $\F'$ of pairwise intersections. Let $k=\lfloor \frac{r}{2}\rfloor$. By Theorem \ref{main2}, there exist $k+1$ lines $\ell_0,\dots,\ell_k$ in $\R^2$ piercing $\F'$. Write $L = \bigcup_{i=0}^k \ell_i$.

For $S\in \F$, let $\bar{S} = S \cap L$. Then $\bar{\F} = \{\bar{S} \mid S\in \F \}$ is a family of $(k+1)$-intervals. We claim that $\bar{\F}$ has the $(p+1,2)$ property. Indeed, suppose for contradiction that there exist pairwise disjoint  sets $\bar{S}_1,\dots, \bar{S}_{p+1} \in \bar{\F}$.
Since $\F$ has the $(p+1,2)$ property,  there exist integers $1\le q<q'  \le p+1$ so that $S_q \cap S_{q'} \neq \emptyset$. But then  $S_q \cap S_{q'} \in \F'$, which means that there exists some  line $\ell_i$, so that $S_j \cap S_k \cap \ell_i \neq \emptyset$. However, this contradicts  $\bar{S_q} \cap \bar{S_{q'}} =\emptyset$. 

If $k=0$ then $r=1$ and $\bar{\F}$ is a family of intervals. Therefore, by Theorem \ref{gallai}, we have $\tau(\bar{\F}) \le  p$. If $k\ge 1$, then by Theorem \ref{dintervals} we have $\tau(\bar{\F}) \le  ((k+1)^2-(k+1))p = (\lfloor \frac{r}{2} \rfloor ^2 +\lfloor\frac{r}{2} \rfloor)p$. 

Finally, we claim that $\tau (\F) \le \tau(\bar{\F})$, which will conclude the proof of the theorem. Indeed, let $S\in \F$. Since $S$ is not isolated, there exists some $S'\in \F$ with $S\cap S' \neq \emptyset$, and moreover, there exists some $i\in \{0,\dots,k\}$ with $S\cap S'\cap \ell_i \neq \emptyset$. This implies that $\bar{S} \neq \emptyset$ for all $S\in \F$, entailing the claim. \qed

\section*{Acknowledgment} The author is grateful to Joseph Miller for his help with finding the example after Theorem \ref{main2}. The author is also grateful to Daniel McGinnis for useful comments and for his help with generating the figures.

\end{document}